\documentclass[11pt]{amsart}

\usepackage{amssymb}
\usepackage{amsfonts}
\usepackage{tikz}
\usepackage{xcolor}
\usepackage{mathtools}
\usepackage{graphicx}
\usepackage{enumerate}
\usepackage{amsmath,amscd}
\usepackage{amsmath,amsthm,amssymb,anysize}
\usepackage[numbers,sort&compress]{natbib}
\usepackage{color}
\usepackage{graphicx}
\usepackage[all,cmtip]{xy}
\usepackage[colorlinks,citecolor=blue]{hyperref}
\usepackage{tikz-cd} 
\usepackage{blindtext}
\usepackage[normalem]{ulem}
\theoremstyle{definition}
\newtheorem{thm}{Theorem}
\newtheorem{lem}[thm]{Lemma}

\newtheorem{prop}[thm]{Proposition}
\newtheorem{cor}[thm]{Corollary}

\newtheorem{defn}{Definition}

\newcommand{\bea}{\begin{eqnarray}}
\newcommand{\eea}{\end{eqnarray}}

\newcommand{\bfx}{\mathbf{x}}
\newcommand{\opn}{\operatorname}
\newcommand{\tor}{\operatorname{Tor}}
\newcommand{\fm}{\mathfrak{m}}

\newcommand{\abs}[1]{\lvert#1\rvert}
\marginsize{3cm}{3cm}{2.5cm}{2.2cm}
\numberwithin{equation}{section}
\numberwithin{thm}{section}
\def\Z{\mathbb{Z}}

\def\C{\mathbb{C}}

\def\F{\mathbb{F}}
\def\H{\mathrm{H}}

\def\PU{\mathrm{PU}}

\def\Aut{\mathrm{Aut}}

\def\U{\mathrm{U}}

\def\CP{\mathbb{CP}}

\def\gm{\Gamma}

\def\Eb{E_{\text{btg}}}
\def\Ef{E_{\text{flex}}}
\def\El{E_{\text{line}}}

\def\Bb{B_{\text{btg}}}
\def\Bf{B_{\text{flex}}}
\def\Bl{B_{\text{line}}}

\def\epl{\epsilon}
\def\bitpro{\mathbf{{Bitangent}}(\epl)}
\def\flexpro{\mathbf{Flex}(\epl)}
\def\linepro{\mathbf{Line}(\epl)}

\makeatletter
\newcommand{\extp}{\@ifnextchar^\@extp{\@extp^{\,}}}
\def\@extp^#1{\mathop{\bigwedge\nolimits^{\!#1}}}
\makeatother
\usepackage{mathtools}

\title[Topological complexity of enumerative problems and $B\PU_n$]{Topological complexity of enumerative problems and classifying spaces of $\PU_n$}

\author{Weiyan Chen}
\address{Weiyan Chen, Yau Mathematical Sciences Center, Tsinghua University, Shuangqing Building Complex A, Haidian district, Beijing, China.}
\email{chwy@tsinghua.edu.cn}

\author{Xing Gu}

\address{Xing Gu, Institute for Theoretical Sciences, School of Science, Westlake University, 600 Dunyu Road, Sandun town, Xihu district, Hangzhou 310030, Zhejiang Province, China.}
\email{guxing@westlake.edu.cn}
\thanks{}

\subjclass[2020]{55M30, 55R80, 55R40}

\keywords{topological complexity, Schwarz genus, cubic surfaces, quartic curves, projective unitary groups}

\begin{document}

\begin{abstract}
We study the topological complexity, in the sense of Smale, of three enumerative problems in algebraic geometry: finding the 27 lines on cubic surfaces, the 28 bitangents and the 24 inflection points on quartic curves. In particular, we prove lower bounds for the topological complexity of any algorithm that finds solutions to the three problems and for the Schwarz genera of their associated covers. The key is to understand cohomology classes of the classifying spaces of  projective unitary groups $\PU_n$. 
\end{abstract}

\maketitle

\section{Introduction}
Enumerative geometry is concerned with counting the number of solutions to geometric problems. For example, it is classically known that over $\C$
\begin{enumerate}
    \item any smooth cubic surface has 27 lines;
    \item any smooth quartic plane curve has 28 bitangent lines;
    \item any generic quartic plane curve has 24 inflection points.
\end{enumerate}
After knowing the number of solutions, a natural   question is: How complex are the solutions to these problems? In this paper, we study the \emph{topological complexity}, in the sense of Smale \cite{Smale}, of the three enumerative problems listed above.

Following Smale, by an \emph{algorithm} we mean a finite rooted tree consisting of a root for the input, leaves for the output, and internal nodes of the following two types:
    \begin{quote}
    \begin{center}
    \emph{computation nodes}\ \ \begin{tikzpicture}[scale=0.3]
    \draw (0,0) circle(.2);
    \draw[->] (0,1)-- (0,0.1);
    \draw[->] (0,-0.1)-- (0,-1);
    \end{tikzpicture}
    \qquad and \qquad 
    \emph{branching nodes}\ \  \begin{tikzpicture}[scale=0.3]
    \draw (0,0) circle(.2);
    \draw[->] (0,1)-- (0,0.1);
    \draw[->] (-0.08,-0.08)-- (-0.8,-0.8);
    \draw[->] (0.08,-0.08)-- (0.8,-0.8);
    \end{tikzpicture}\ \ .
        \end{center}
\end{quote}
The \emph{topological complexity} of an algorithm is the number of branching nodes in the tree. The {topological complexity of a problem $P$} is the minimum of the topological complexity of any algorithm solving $P$. Smale \cite{Smale} first proved a lower bound for the topological complexity of the problem of finding roots of a polynomial $f(z)=0$. Smale's lower bounds were later improved by Vassiliev \cite{Vassiliev},  De Concini-Procesi-Salvetti \cite{DPS}, and Arone \cite{Arone}. After results about finding roots of polynomials which are 0-dimensional objects, it is natural to ask similar question about solutions to enumerative problems for high-dimensional objects such as curves and surfaces,  which is mostly unknown. A first step in this direction is the work by the first author and Wan \cite{Chen-Wan} concerning the topological complexity of finding inflection points on cubic plane curves. In this paper, we consider the following problems: 
\begin{itemize}
    \item $\linepro$: Given any cubic surface defined by a homogeneous polynomial $F(x,y,z,w)$ of degree 3, find all of its 27 lines $(l_1,\cdots,l_{27})$ within $\epl$.
    \item $\bitpro$: Given any quartic curve defined by a homogeneous polynomial $F(x,y,z)$ of degree 4, find all of its 28 bitangent lines $(l_1,\cdots,l_{28})$ within $\epl$.
    \item $\flexpro$: Given any quartic curve defined by a homogeneous polynomial $F(x,y,z)$ of degree 4, find all of its 24 inflection points $(p_1,\cdots,p_{28})$ within $\epl$.
\end{itemize}
An algorithm is said to solve the problem $\linepro$ if it inputs an arbitrary homogeneous cubic polynomial $F(x,y,z,w)$ and outputs a sequence of lines $(l_1',\cdots,l_{27}')$ such that each $l_i'$ is $\epl$-close to a lines $l_i$ where $l_i$'s are the 27 distinct lines on the cubic surface defined by $F$. Here, the ``within $\epl$" condition can be measured by any metric on the Grassmannian consisting of lines in $\CP^3$ inducing the standard topology. Similar conditions hold for $\bitpro$ and $\flexpro$.

\begin{thm}
    \label{TC main theorems}
    When $\epl$ is sufficiently small, we have
    \begin{enumerate}
    \item the topological complexity of the problem $\linepro$ is at least 15,
    \item the topological complexity of the problem $\bitpro$ is at least 8, 
    \item the topological complexity of the problem $\flexpro$ is at least 8.
    \end{enumerate}
\end{thm}

The lower bounds in Theorem \ref{TC main theorems} come from the following theorems about the topology of various parameter spaces associated to the three enumerative problems. First, consider 
\begin{align*}
    \Bl&:=\{\text{nonsingular homogeneous cubic polynomials $F(x,y,z,w)$}\}/\C^\times\\
    \El&:=\{(F,l_1,\cdots, l_{27}) :  \text{ $l_i$'s are the 27 lines on the cubic surface $F\in \Bl$}\}.
\end{align*}
The space $\Bl$ is an open submanifold of $\CP^{19}$ consisting of all homogeneous cubic polynomials in four variables. The projection $\El\to \Bl$ given by $(F,l_1,\cdots, l_{27})\mapsto F$ is a normal $S_{27}$-cover, where $S_{27}$ acts on $\El$ by permuting the ordering of the 27 lines.
\begin{defn}[\cite{Schwarz}]
The \emph{Schwarz genus} of a covering $E\to B$, denoted by $g(E\to B)$ or simply $g(E)$, is the minimum size of an open cover of $B$ consisting of open sets such that there exists a continuous section of the covering map $E\to B$ over each open set. 
\end{defn}
\begin{thm}
\label{SG line}
    $g(\El\to \Bl)\ge 16$.
\end{thm}
Similarly, consider 
\begin{align*}
    \Bb&:=\{\text{nonsingular homogeneous quartic polynomials $F(x,y,z)$}\}/\C^\times\\
    \Eb&:=\{(F,l_1,\cdots, l_{28}) :  \text{ $l_i$'s are the 28 bitangent lines to the curve $F\in \Bb$}\} 
\end{align*}
The space $\Bb$ is an open submanifold of $\CP^{14}$ consisting of all homogeneous quartic polynomials in three variables. The projection $\Eb\to \Bb$ gives a normal $S_{28}$-cover.

\begin{thm}
\label{SG bitangent}
    $g(\Eb\to \Bb)\ge 9$.
\end{thm}

A generic quartic curve has 24 inflection points. However, unlike in the previous examples, being smooth is not enough for a quartic curve to have 24 distinct inflection points because two inflection points can coincide. Hence, we consider
\begin{align*}    
    \Bf&:=\{F(x,y,z) : F \text{ defines a quartic curve with 24 distinct inflection points}\}/\C^\times\\
    \Ef&:=\{(F,p_1,\cdots, p_{24}) :  \text{ $p_i$'s are the 24 distinct inflection points on $F\in \Bf$}\}.    
\end{align*}
$\Bf$ is an open submanifold of $\CP^{14}.$ The  projection $\Ef\to\Bf$ gives a normal $S_{24}$-cover. 
\begin{thm}
\label{SG flex}
    $g(\Ef\to \Bf)\ge 9 $.
\end{thm}
Let us briefly outline the proof of Theorems \ref{SG bitangent} and \ref{SG flex}. The parameter space of quartic curves admits a natural action of the projective unitary group $\PU_3=\U_3/Z(\U_3)$. The lower bound in Theorem \ref{SG flex} is obtained by restricting the cover $\Ef\to\Bf$ to a subspace of $\Bf$ given by the $\PU_3$-orbit of the the Klein quartic curve defined by  $x^3y+y^3z+z^3x=0$. This orbit is homeomorphic to the homogeneous space $\PU_3/\Aut(F)$ where $\Aut(F)$ is the automorphism group of the Klein quartic curve $F$. The most difficult part of this proof is to find nonzero obstruction classes in the cohomology of this homogeneous space using results about the cohomology of the classifying spaces of $\PU_3$. Similarly, to prove Theorems \ref{SG line}, we consider the Fermat cubic surface $x^3+y^3+z^3+w^3=0$ and its $\PU_4$-orbit. This proof is rather flexible and can be modified to study the topological complexity of enumerative problems other than the three we considered in this paper.

Can the inequalities in the theorems above be equalities? We don't expect so. For example, $\Bl$ is an open submanifold of $\CP^{19}$ and hence by the Andreotti–Frankel theorem is homotopy equivalent to a CW complex of dimension 19. For dimension reasons, we have that $g(\El\to \Bl)\le 20$ by Corollary \ref{cor:coho dim} below. Similarly, we can show  that  $g(\El\to \Bl)\le 15$ and $g(\Ef\to\Bf)\le 15$. In general, it is a difficult problem to determine the exact value of Schwarz genus. Even for the most well-studied example, the Schwarz genus of the $S_n$-cover of the configuration space of $n$ distinct unordered points in $\C$ has various upper and lower bounds thanks to the works of Smale \cite{Smale}, Vassiliev \cite{Vassiliev}, De Concini-Procesi-Salvetti \cite{DPS}, and Arone \cite{Arone}, but its exact value remains unknown for infinitely many integers $n$. The main point of this paper is to broaden the domain of study of topological complexity and Schwarz genus, from finding roots of polynomials to finding solutions of enumerative problems about curves and surfaces. We leave the problem of improving the bounds to future works.

The cohomology of the classifying space of $\PU_n$, which is the main technical tool in this paper, is not yet fully understood despite a half century of research. Kono-Mimura \cite{kono1975cohomology}, Toda \cite{toda1987cohomology}, Vezzosi \cite{vezzosi2000chow}, and Vistoli \cite{vistoli2007cohomology} have made significant contributions.    
More recently, the second author \cite{gu2019cohomology} determines the cohomology groups $H^k(B\PU_n;\Z)$ for $k\leq 10$, which is a key ingredient of proofs in this paper. In contrast to the cohomology classes of $B\U_n$, \emph{i.e.} the universal Chern classes, which have rich geometric interpretations and wide applications to geometric topology and algebraic geometry, the cohomology classes of $B\PU_n$ do not have as many geometric interpretations and applications. For this reason, we believe that our results would be an interesting addition to the literature about the cohomology of $B\PU_n$ which has been active in recent years. 

\section*{Acknowledgment}
We thank Yichang Cai, Lifan Guan, and Yingjie Tan for helpful conversations. WC is partially supported by the Young Scientists Fund of the National Natural Science Foundation of China (Grant No. 12101349). XG is partially supported by the Young Scientists Fund of the National Natural Science Foundation of China (Grant No.12201503).

\section{Preliminary results}
We first recall and prove some general results that will be used later in our proofs. 

\subsection{General results about Schwarz genus}
We recall some general results about Schwarz genus, mostly due to A. Schwarz \cite{Schwarz}. We assume all spaces to be normal and Hausdorff. We also assume each open cover to have a partition of unity subordinated to it. This assumption holds if the space is paracompact. 

\begin{prop}[\cite{Schwarz}, p.71]\label{pro:pullback SG}
Let $i^*E\to B'$ denote the pullback of a cover $E\to B$ along a continuous map $i: B'\to B$. Then $g(i^*E\to B')\le g(E\to B)$.
\end{prop}

\begin{prop}[\cite{Schwarz}, p.71]\label{pro:fiberwise join}
Consider a covering 
$E\to B$ with fiber $F$. Let $$F^{*k}\to E_k\to B$$ denote the fiberwise join bundle where $F^{*k}$ is the topological join of $F$ with itself $k$ times.
Then $g(E\to B)\le k$ if and only if $E_k\to B$ has a continuous section.
\end{prop}

\begin{cor}[\cite{Schwarz}, p.76]
\label{cor:coho dim}
If $B$ is a CW complex of dimension $d$, then $g(E\to B)\le d+1.$
\end{cor}
\begin{proof}[Proof (Proposition \ref{pro:fiberwise join} $\Rightarrow$ Corollary \ref{cor:coho dim})]
The obstruction to extending a section of $E_{d+1}\to B$ to the $i$-skeleton of $B$ is an element in $H^{i+1}(B;\pi_i(F^{*(d+1)}))$. This cohomology group vanishes for $i\ge d$ since $B$ has cohomological dimension $d$. If $i\le d-1$, then $\pi_i(F^{*(d+1)})=0$ since $F^{*(d+1)}$ is always $(d-1)$-connected. Hence there is no obstruction to a section.
\end{proof}


\begin{prop}[\cite{Schwarz}, p.98]\label{pro:homo genus}
Suppose that $E\to B$ is a normal $\Gamma$-cover (or equivalently, a principal $\Gamma$-bundle) with a classifying map $cl:B\to B\Gamma$ where $B\gm$ is the classifying space of $\Gamma$. Then we have 
$$g(E\to B)\ge \min\{k:\H^i(B\Gamma;A)\xrightarrow{cl^*}\H^i(B;A)\text{ is zero for any }i\ge k\}$$
for any $\Gamma$-module $A$. The integer on the right hand side of the inequality above is called the \emph{homological $A$-genus} of the cover $E\to B.$
\end{prop}

\begin{prop}[Disconnected covers]
\label{pro:disconnected covers}
Consider a cover $E\to B$ with path-components $E=\bigcup_{i\in I} E_i$ where each $E_i\to B$ is also a cover. Suppose that there exists an $m\in I$ such that for any $i\in I$, there exists morphism of coverings $E_i\to E_m$. Then $g(E)=g(E_m).$
\end{prop}
\begin{proof}
In general, a morphism of coverings $E\to E'$ implies that $g(E')\le g(E)$, because any section of $E$ composes to a section of $E'.$ In our situation, we have two morphisms of covers
$$E_m\hookrightarrow E=\bigcup_{i\in I} E_i \longrightarrow E_m$$
where the first map is the natural inclusion and the second map is the disjoint union of morphisms $E_i\to E_m$ for all $i\in I$. Hence we have $g(E_m)\le g(E)\le g(E_m).$
\end{proof}

\subsection{On the cohomology of $B\PU_n$}\label{sec:BPU}
Let $f\colon B\U_n\to B\PU_n$ be the map induced by the natural projection $\U_n\to \PU_n$. 
The map $f$ is part of a fibration
\begin{equation}\label{eq:hofiber}
    B\U_n\xrightarrow{f} B\PU_n\to K(\Z,3).
\end{equation}
Let $(E, d)$ denote the associated cohomological Serre spectral sequence with coefficients in $R$ which is either $\Z$ or $\Z_{(p)}$ for a prime number $p$. We know that $H^*(B\U_n;R)\cong R[c_1,\cdots,c_n]$ where $c_i$ is the $i$-th Chern class. Let $x$ denote the canonical generator of $H^3(K(\Z,3); R)$.  Corollary 3.4 in \cite{gu2019cohomology} tells us that
\begin{equation}\label{eq:d3}
    d_3^{0,2k}(c_k) = \nabla(c_k)\otimes x
\end{equation}
where $\nabla:H^*(BU_n;R)\to H^{*-2}(BU_n;R)$ is an additive function defined by
\begin{equation*}
    \nabla(c_k) = (n-k+1)c_{k-1}
\end{equation*}
together with the Leibniz rule. Furthermore, by the proof of Theorem 1.3 in \cite{gu2019cohomology}, we have
\begin{equation}\label{eq:E4Einfty}
   \opn{Ker}d_3^{0,k} = E_4^{0,k} = E_{\infty}^{0,k} = \opn{Im}(f^*)^k
\end{equation}
for $k\leq 12$ and $R = \Z$, or equivalently, the sequence
\begin{equation}\label{eq:nablaSeq}
    H^k(BPU_n;\Z)\xrightarrow{f^*} H^k(BU_n;\Z)\xrightarrow{\nabla}H^{k-2}(BU_n;\Z)
\end{equation}
is exact for $k\leq 12$. Since $\Z_{(p)}$ is a flat $\Z$-module, the exactness of \eqref{eq:nablaSeq} is preserved when we take $R = \Z_{(p)}$. 
After routine computations of $\opn{Ker}\nabla = \opn{Im}f^*$ using the definition above, we obtain the following results.

\begin{prop}\label{pro:n=3}
    When $n = 3$ and $R = \Z$ or $R = \Z_{(p)}$ for a prime $p$, the image $\opn{Im}f^*$ has the following generators
\begin{equation}
    \begin{split}
        \opn{dim}2\colon & 0,\\
        \opn{dim}4\colon & c_1^2 - 3 c_2,\\
        \opn{dim}6\colon & 2c_1^3 - 9 c_1 c_2 + 27 c_3.\\
    \end{split}
\end{equation}
\end{prop}

\begin{prop}\label{pro:n=4}
    When $n = 4$ and $R = \Z$ or $R = \Z_{(p)}$ for a prime $p$, the image $\opn{Im}f^*$ has the following generators
\begin{equation}
    \begin{split}
        \opn{dim}2\colon & 0,\\
        \opn{dim}4\colon & 3 c_1^2 - 8 c_2,\\
        \opn{dim}6\colon & c_1^3 - 4 c_1 c_2 + 8 c_3,\\
        \opn{dim}8\colon & (3 c_1^2 - 8 c_2)^2,\ 3 c_1^4 - 16 c_1^2 c_2 + 64 c_1 c_3 - 256 c_4.
    \end{split}
\end{equation}
\end{prop}

\begin{lem}\label{lem:ImageModp}
    Let $p$ be a prime number not dividing $n$. Let $g\colon BSU_n\to BU_n$ be the map induced by the standard inclusion. Then we have the following commutative diagram
    \begin{equation}\label{eq:ImageModpDiagram}
        \begin{tikzcd}
            H^k(BPU_n;\Z_{(p)})\arrow[d]\arrow[r,hook,"f^*"]&H^k(BU_n;\Z_{(p)})\arrow[d]\arrow[r,two heads,"g^*"] &H^k(BSU_n;\Z_{(p)}) \arrow[d,two heads]\\
            H^k(BPU_n;\F_p)\arrow[r,hook, "f^*"]          &H^k(BU_n;\F_p)\arrow[r,two heads,"g^*"]     &H^k(BSU_n;\F_p)
        \end{tikzcd}
    \end{equation}
    where the hooked arrows are injective and the double-headed arrows are surjective. 
    Furthermore, the following two vector subspaces are equal:
    \begin{equation}\label{eq:ImageModp_1}
        \opn{Im}\{H^*(BPU_n;\Z_{(p)})\to H^*(BU_n;\Z_{(p)})\to H^*(BU_n;\F_p)\}
    \end{equation}
    and
    \begin{equation}\label{eq:ImageModp_2}
        \opn{Im}\{H^*(BPU_n;\F_p)\to H^*(BU_n;\F_p)\}.
    \end{equation}
\end{lem}
\begin{proof}
     Since $p\nmid n$, the compositions $g^*f^*$ of both rows of \eqref{eq:ImageModpDiagram} are isomorphisms. Therefore the horizontal hooked arrows and double-headed arrows are justified. The vertical double-headed one is well known. 
     
     Both \eqref{eq:ImageModp_1} and \eqref{eq:ImageModp_2} are $\F_p$-vector spaces, and \eqref{eq:ImageModp_1} is a subspace of \eqref{eq:ImageModp_2}. Therefore, it suffices to show that they have the same dimensions.     Since $g^*f^*$ of the second row is an isomorphism, the dimension of $\eqref{eq:ImageModp_2}$ equals $\opn{dim}H^k(BSU_n;\F_p)$. Since $g^*f^*$ of the first row is an isomorphism, by diagram-chasing we have the $H^k(BPU_n;\Z_{(p)})$ at the upper-left corner maps surjectively onto the $H^*(BSU_n;\F_p)$ at the lower-right corner. Therefore, the dimension of $\eqref{eq:ImageModp_1}$ is no less than $\opn{dim}H^k(BSU_n;\F_p)$. This concludes the proof.
\end{proof}

From Lemma~\ref{lem:ImageModp} and Proposition~\ref{pro:n=3} we deduce the following
\begin{prop}\label{pro:n=3generators}
    For a prime number $p\ne 3$, there exist cohomology classes 
    \begin{equation*}
        \epl_k\in H^k(B\PU_3;\F_p),\ k = 4,6,
    \end{equation*}
    satisfying
    \begin{equation*}
        H^*(B\PU_3;\F_p) = \F_p[\epl_4,\epl_6],
    \end{equation*}
   such that $f^*:H^*(B\PU_3;\F_p)\to H^*(B\U_3;\F_p)$ maps
    \begin{equation}\label{eq:n=3generators}
        \begin{split}
            &f^*(\epl_4) = c_1^2 - 3 c_2,\\
            &f^*(\epl_6) = 2c_1^3 - 9 c_1 c_2 + 27 c_3 .\\  
        \end{split}
    \end{equation}
\end{prop}

From Lemma~\ref{lem:ImageModp} and Proposition~\ref{pro:n=4} we deduce the following
\begin{prop}\label{pro:n=4generators}
    For a prime number $p\ne 2$, there exist cohomology classes 
    \begin{equation*}
        \epl_k\in H^k(B\PU_4;\F_p),\ k = 4,6,8,
    \end{equation*}
       such that $f^*:H^*(B\PU_4;\F_p)\to H^*(B\U_4;\F_p)$ maps
    \begin{equation*}
        H^*(B\PU_4;\F_p) = \F_p[\epl_4,\epl_6,\epl_8],
    \end{equation*}
    and
    \begin{equation}\label{eq:n=4generators}
        \begin{split}
            &f^*(\epl_4) = 3 c_1^2 - 8 c_2 ,\\
            &f^*(\epl_6) = c_1^3 - 4 c_1 c_2 + 8 c_3,\\
            &f^*(\epl_8) = 3 c_1^4 - 16 c_1^2 c_2 + 64 c_1 c_3 - 256 c_4.
        \end{split}
    \end{equation}
\end{prop}

\subsection{Cohomology of $\PU_n$ homogeneous spaces}\label{sse:homo space}

Consider a subgroup $\Gamma\subseteq\PU_n$. In order to use Proposition \ref{pro:homo genus} to calculate the homological genus of the principal $\Gamma$-bundle $\PU_n\to \PU_n/\Gamma$, we need to study its classifying map $cl: \PU_n/\Gamma\to B\Gamma$ and the induced map $cl^*$ on cohomology. Observe that the map $cl$ is a part of a fibration
\begin{equation}
\label{eq:PU/Gamma fibration}
        PU_n/\Gamma\xrightarrow{cl} B\Gamma\xrightarrow{\varphi}BPU_n
\end{equation}
where $\varphi: B\Gamma\to B\PU_n$ denotes the map induced by the inclusion $\Gamma\hookrightarrow\PU_n$.

\begin{defn}
    \label{def:regular seq}
    For a graded commutative unital ring $R$ and a graded $R$-module $M$, a sequence $\bfx = (x_1,\cdots,x_k)$ of elements in $R$ is called a \textit{regular sequence in} $M$ if 
 \begin{enumerate}
     \item multiplication by $x_1$ is an injective endomorphism on $M$, and for each $i > 1$, multiplication by $x_i$ is an injective endomorphism on $M/(x_1,\cdots,x_{i-1})M$, and
     \item $(x_1,\cdots,x_k)\neq M$.
 \end{enumerate}
\end{defn}

\begin{prop}
    \label{pro:PUn/Gamma}
    Let $\Gamma$ be a subgroup of $\PU_n$ which is isomorphic to $(\Z/q\Z)^{\times m}$ where $q$ is a power of a prime number $p\nmid n$. Let $\epl_4,\cdots,\epl_{2n}$ be generators of $H^*(B\PU_n;\F_p)$. Suppose that 
    $\varphi^*:H^*(B\PU_n;\F_p)\to H^*(B\Gamma;\F_p)$
    satisfies that  
    \begin{itemize}
        \item  when $(p,q)=(2,2)$, then $(\varphi^*(\epl_4),\cdots,\varphi^*(\epl_{2n}))$  is a regular sequence in
        $$    H^*(B\Gamma;\F_2)\cong \F_2[v_1,\cdots,v_m], \qquad\abs{v_i} = 1;
        $$
    \item  when $(p,q)\ne (2,2)$, then $(\varphi^*(\epl_4),\cdots,\varphi^*(\epl_{2n}))$ is a regular sequence in $\F_p[\xi_1,\cdots,\xi_m]$, regarded as a subring of 
    $$    H^*(B\Gamma;\F_p)\cong \F_p[\xi_1,\cdots,\xi_m]\otimes_{\F_p}\Lambda_{\F_p}[u_1,\cdots,u_m],\qquad \abs{\xi_i} = 2,\ \abs{u_i} = 1.$$
    \end{itemize}
    Then the homomorphism 
    \[cl^*\colon H^*(B\Gamma;\F_p)\to H^*(PU_n/\Gamma;\F_p)\]
    is surjective and has kernel the ideal generated by $\varphi^*(\epl_4),\cdots,\varphi^*(\epl_{2n})$.
\end{prop}

To prove Proposition \ref{pro:PUn/Gamma}, we will first reduce it to a result about polynomial algebras. 

The Eilenberg-Moore spectral sequence associated to the fibration (\ref{eq:PU/Gamma fibration}) is of the form
\begin{equation}
\label{eq:E2 Eilen-Moore}
    E_2^{s,t}\cong\tor^{s,t}_{H^*(B\PU_n;\F_p)}\big(\F_p, H^*(B\Gamma;\F_p)\big)\Rightarrow H^{s+t}(\PU_n/\Gamma;\F_p).
\end{equation}
There are different conventions in literature for the $(s,t)$-grading in the Eilenberg-Moore spectral sequence. Here we follow the convention in \cite{mccleary2001user}, Chapter 7 which we recall now. For a graded commutative unital ring $R$ and graded $R$-modules $M,N$, let
\begin{equation*}
    \cdots\to P^{-n}\to\cdots\to P^{-1}\to P^0\to M
\end{equation*}
be a projective $R$-resolution of $M$, where each $P^{-n}$ is a projective graded $R$-module. The total degree of an element of $P^{-n}\otimes_R N$ is called the \textit{internal degree}.  The non-positive integer $-n$ is the \textit{homological degree}. Then we have the derived functor
\begin{equation}
\label{eq:def of tor}
    \tor_R^{s,t}(M,N)\cong H_{-s,t}(P^*\otimes_R N)
\end{equation}
where $s$ and $t$ denote the homological degree and the internal degree, respectively.  In particular, we have that  $E_2^{s,t} = 0$ unless $s\leq 0$ and $t \geq 0$ and that the differentials take the form
\begin{equation}\label{eq:EM_diff}
    d_r^{s,t}\colon E_r^{s,t}\to E_r^{s+r, t-r+1}.
\end{equation}

Next, to calculate the $E_2$-page as a Tor-functor, we need to start with a resolution of $\F_p$ as a module over the ring 
$$H^*(B\PU_n;\F_p)\cong \F_p[\epsilon_{4},\epsilon_6,\cdots,\epl_{2n}].$$
One such resolution is given by the Koszul complex which we now briefly recall. Let $R$ be a graded commutative unital ring. Consider a finite sequence of elements in $R$, denoted by $\bfx = (x_1,\cdots,x_n)$. The \textit{Koszul complex of $\bfx$}, denoted 
by $K(\bfx)$, is a chain complex of $R$-modules with
\begin{equation}
    \label{eq:def Koszul complex}
    K(\bfx)_i = \extp^i R^n
\end{equation}
together with boundary operators $d_i:K(\bfx)_{i}\to K(\bfx)_{i-1}$ defined in the following way. Let $\{e_i : 1\leq i\leq n\}$ denote the standard basis of $R^n$. Then $K(\bfx)_i$ has a basis of the form 
 \begin{equation*}
     \{e_{j_1}\wedge\cdots\wedge e_{j_i}\ :\ 1\leq j_1<j_2\cdots<j_i\leq n\}.
 \end{equation*}
The boundary operator $d_{i}$ is an $R$-linear map such that
 \begin{equation*}
     d_i(e_{j_1}\wedge\cdots\wedge e_{j_i}) = \sum_{k=1}^i (-1)^{k+1}x_{j_k}  e_{j_1}\wedge\cdots\Hat{e}_{j_k}\cdots\wedge e_{j_i}.
 \end{equation*}
We record the following two standard facts about the Koszul complex.
\begin{prop}[special case of Corollary 4.5.5 in \cite{weibel1994introduction}]
\label{pro:resolution}
Let $R = \F[x_1,\cdots,x_n]$ be a polynomial algebra over a field $\F$. Let $\bfx = (x_1,\cdots,x_n)$. Define $\varepsilon\colon K(\bfx)_0 = R\to\F$ by
 \[\varepsilon(f(x_1,\cdots,x_n)) = f(0,\cdots,0).\]
Then the chain complex 
    \[\cdots\xrightarrow{d}K(\bfx)_i\xrightarrow{d}\cdots\to K(\bfx)_0\xrightarrow{\varepsilon}\F\to 0\]
    is a projective resolution of $\F$ as an $R$-module. 
 \end{prop}
\begin{prop}[Corollary 4.5.4 in \cite{weibel1994introduction}]\label{pro:Koszul}
    If $\bfx$ is a regular sequence on $M$, then we have
    \begin{equation*}
        H_i(K(\bfx)\otimes_R M) = 
        \begin{cases}
            M/\bfx M,\ i = 0,\\
            0,\ i > 0.
        \end{cases}
    \end{equation*}
\end{prop}

Now we can finish the proof of Proposition \ref{pro:PUn/Gamma}.
\begin{proof}[Proof of Proposition \ref{pro:PUn/Gamma}]
First consider the case when $(p,q)\ne(2,2)$. Let us abbreviate:
\begin{align*}    
    R&:=H^*(B\PU_n;\F_p)\cong \F_p[\epsilon_{4},\epsilon_6,\cdots,\epl_{2n}], &|\epl_{i}|=i,\\
    P&:= \F_p[\xi_1,\cdots,\xi_m], &\abs{\xi_j} = 2,\\
    Q& := \Lambda_{\F_p}[u_1,\cdots,u_m], &\abs{u_j} = 1,\\
    S&:=H^*(B\Gamma;\F_p)\cong P\otimes_{\F_p}Q.
\end{align*}
  Then $\varphi^*:R\to S$ makes $S$ an $R$-module. 
  Let $\epl$ denote the sequence
    \[(\epl_4,\epl_6,\cdots,\epl_{2n})\]
    and $\varphi^*(\epl)$ the sequence
    \begin{equation}
        \label{eq:phi of epsilon}
        (\varphi^*(\epl_4),\varphi^*(\epl_6),\cdots,\varphi^*(\epl_{2n})).
    \end{equation}
    By the definition of Koszul complex \eqref{eq:def Koszul complex} and by the explicit presentations of $P$, $Q$, $R$ and $S$, we have the following isomorphism of chain complexes of $S$-modules
    \begin{equation}\label{eq:EN_Koszul}
        K(\epl)\otimes_R S\cong K(\epl)\otimes_R (P\otimes_{\F_p} Q)\cong
        K(\varphi^*(\epl))\otimes_{\F_p}Q
    \end{equation}
Now we can calculate the $E_2$-page of the Eilenberg-Moore spectral sequence \eqref{eq:E2 Eilen-Moore}:
\begin{align*}
    E_2^{s,t}&\cong \tor^{s,t}_{R}(\F_p, S) 
     \cong H_{-s,t}(K(\epl)\otimes_R S) &\text{by \eqref{eq:def of tor} and Proposition \ref{pro:resolution}}\\
    &\cong H_{-s,t}\Big(K\big(\varphi^*(\epl)\big)\otimes_{\F_p}Q\Big) &\text{by \eqref{eq:EN_Koszul}}\\
    &\cong H_{-s,t}\Big(K\big(\varphi^*(\epl)\big)\Big)\otimes_{\F_p}Q &\text{since $Q$ is a flat $\F_p$-module}\\
    &\cong\begin{cases}
            H^t(B\Gamma)/(\varphi^*(\epl_4),\cdots,\varphi^*(\epl_{2n})),\ s = 0, \\
            0,\ s\neq 0
        \end{cases}
        &\text{by Proposition \ref{pro:Koszul}}
\end{align*}
The hypothesis that $\varphi^*(\epl)$ is a regular sequence in $P$ was used in the last isomorphism. Since $E_2$ is zero off the vertical axis, all differentials must be zero and $E_2 = E_\infty$. The edge morphism of the spectral sequence gives the following commutative diagram:
\[\begin{tikzcd}
	{S=H^*(B\Gamma;\F_p)} & {H^*(\PU_n/\Gamma;\F_p)} \\
	{\F_p\otimes_R S=E_2^{0,*}} & {E_\infty^{0,*}}
	\arrow["{cl^*}", from=1-1, to=1-2]
	\arrow[two heads, from=1-1, to=2-1]
	\arrow["{=}"', from=2-1, to=2-2]
	\arrow[hook, from=2-2, to=1-2]
\end{tikzcd}\]
Since $E_\infty^{s,t}=0$ when $s\ne0$, the second vertical inclusion in the diagram above is in fact an isomorphism $H^*(\PU_n/\Gamma;\F_p)\cong E_\infty^{0,*}$. In particular, $cl^*$ is surjective. 

If $(p,q)=(2,2)$, then 
$$S:=H^*(B\Gamma;\F_2)\cong \F_2[v_1,\cdots,v_m], \qquad\abs{v_i} = 1.$$
The proof above works in the same way after we replace \eqref{eq:EN_Koszul} by the  isomorphism of chain complexes of $S$-modules $K(\epl)\otimes_R S  \cong K(\varphi^*(\epl))$.
\end{proof}

\section{The 27 lines on cubic surfaces}
This section is devoted to proving Theorem \ref{SG line} about the 27 lines on cubic surfaces. Recall that we defined
\begin{align*}
    \Bl&:=\{\text{nonsingular homogeneous cubic polynomials } F(x,y,z,w)\}/\C^\times\\
    \El&:=\{(F,l_1,\cdots, l_{27}) :  \text{ $l_i$'s are the 27 distinct lines on the cubic surface $F=0$}\}
\end{align*}
The projection  $\pi: \El\to \Bl$ given by $(F,l_1,\cdots, l_{27})\mapsto F$ is a normal $S_{27}$-cover. This cover is not connected since its monodromy representation $\rho:\pi_1(\Bl)\to S_{27}$ is not surjective. The image of $\rho$ is a proper subgroup of $S_{27}$ which is isomorphic to the Weyl group $W(E_6)$. See e.g. \cite{Harris} for more details. 
$\PU_4$
acts naturally on $\CP^3$ and hence also on $\Bl$ and $\El$. The covering map $\pi:\El\to\Bl$ is $\PU_4$-equivariant.

We consider the Fermat cubic surface 
$$F(x,y,z,w)=x^3+y^3+z^3+w^3.$$
Consider the following subgroup $K\le \PU_4$ preserving the Fermat cubic surface $F$:
\begin{equation}
    \label{eq:K27 definition}
    K := \Bigg\langle 
    \begin{bmatrix}
     e^{2\pi i/3} & 0 & 0 & 0\\
    0 & 1 & 0 & 0\\
    0 & 0 & 1 & 0\\    
    0 & 0  & 0 & 1\\    
\end{bmatrix},\
\begin{bmatrix}
     1 & 0 & 0 & 0\\
    0 & e^{2\pi i/3} & 0 & 0\\
    0 & 0 & 1 & 0\\    
    0 & 0  & 0 & 1\\    
\end{bmatrix},\
\begin{bmatrix}
    1 & 0 & 0 & 0\\
    0 & 1 & 0 & 0\\
    0 & 0 & e^{2\pi i/3} & 0\\    
    0 & 0  & 0 & 1\\    
\end{bmatrix}
\Bigg\rangle
\cong (\Z/3\Z)^3. 
\end{equation}

\begin{lem}
\label{Fermat cubic surface automorphism group acting faithfully on 27 lines}
    The action of $K$ on the set of 27 lines on the Fermat cubic surface is faithful. In other words, the induced group homomorphism $K\to S_{27}$ is injective.
\end{lem}
\begin{proof}
    It suffices to check that no nontrivial element in $K$ can simultaneously fix all the 27 lines on the Fermat cubic surface. In fact, no nontrivial element in $K$ can fix the line on the Fermat cubic surface defined by  
    $x+y=z+w=0.$
\end{proof}
The $\PU_4$-orbit of the Fermat cubic surface $F$ gives a well-defined continuous map 
\begin{equation}
    \label{eq:orbit map lines}
\eta: \PU_4/K\to \Bl, \qquad
    \eta(g)=gF.
\end{equation}

\begin{prop}
\label{pro:pullback K disconnected}
Consider the pullback of the cover $\pi:\El\to \Bl$ along the map $\eta$
\[\begin{tikzcd}
	\eta^*\El && \El  \\
	\PU_4/K && \Bl 
	\arrow["{}"', from=1-1, to=2-1]
	\arrow["{\pi}", from=1-3, to=2-3]
	\arrow[""',"\eta", from=2-1, to=2-3]
	\arrow[""', "", from=1-1, to=1-3]
\end{tikzcd}\]
Then the total space of the pullback is homeomorphic to a disjoint union of copies of $\PU_4$:
$$\eta^*\El\cong \bigcup_{S_{27}/K}\PU_4$$
where the components are in bijection with $K$-cosets in $S_{27}$. Here we identify $K$ as a subgroup of $S_{27}$ by Lemma \ref{Fermat cubic surface automorphism group acting faithfully on 27 lines}.
\end{prop}
\begin{proof}
    By definition of pullback, we have
    $$\eta^*\El = \{(gK,l_1,...,l_{27}) : gK\in\PU_4/K \text{ and $l_i$'s are the 27 lines on the surface $gF$}\}.$$
    This space admits a natural action by $\PU_4$ where $h\cdot (gK,l_1,...,l_{27}) = (hgK,hl_1,...,hl_{27})$. Lemma \ref{Fermat cubic surface automorphism group acting faithfully on 27 lines} tells us that this action is free. Hence, $\eta^*\El$ is a disjoint union of $\PU_4$-orbits, where each orbit is a copy of $\PU_4$. Since $\eta^*\El$ is an $S_{27}$-cover of $\PU_4/K$, the components of $\eta^*\El$ must be in bijection of $S_{27}/K$.
\end{proof}
Now we have
\begin{align*}
    g(\El\to\Bl)&\ge g(\eta^*\El\to\PU_4/K)    &\text{by Proposition \ref{pro:pullback SG}}\\
    &=g(\PU_4\to \PU_4/K)&\text{by Proposition \ref{pro:pullback K disconnected} and Proposition \ref{pro:disconnected covers}}
\end{align*}
The proof of Theorem \ref{SG line} will be complete after we prove the following proposition:

\begin{prop}
\label{pro: genus K27}
$g(\PU_4\to\PU_4/K)= 16.$
\end{prop}
\begin{proof}
The upper bound $g(\PU_4\to \PU_4/K)\le 16$ follows from Corollary \ref{cor:coho dim} together with the fact that $\PU_3/K$ is a compact manifold of dimension 15. We focus on the lower bound. 

    In this proof, we will only use cohomology with $\F_3$-coefficients and thus will suppress $\F_3$ in notation. 
    As in the previous section, let $\varphi:BK\to B\PU_4$ denote the map induced by the inclusion $K\hookrightarrow \PU_4$. Notice that the inclusion $K\hookrightarrow \PU_4$ factors as a composition of
    $K\hookrightarrow T\hookrightarrow \U_4\twoheadrightarrow \PU_4$ where $T$ denotes a maximal torus of $\U_4$ consisting of matrices with 0's off diagonal. Therefore, we have
    \begin{equation}
    \label{eq:FFactor}
    \begin{tikzcd}
	H^*(B\PU_4) & H^*(B\U_4) & H^*(BT) & H^*(BK)
	\arrow[from=1-1, to=1-2]
	\arrow["\varphi^*"', bend right=10, from=1-1, to=1-4]
	\arrow[from=1-2, to=1-3]
	\arrow[from=1-3, to=1-4]
    \end{tikzcd}
    \end{equation}
        the composition of which is exactly $\varphi^*$. Proposition \ref{pro:n=4generators} gave 
    \begin{equation}
        \label{eq:CohomologyOfBPU4}
        H^*(B\PU_4)\cong \F_3[\epsilon_4,\epsilon_6,\epsilon_8],\qquad\qquad\qquad
        |\epsilon_i|=i.
    \end{equation}
    We also have the following standard results
\begin{align}
    \label{eq:CohomologyOfBU4}
    &H^*(B\U_4) \cong \F_3[c_1,\cdots,c_4], &|c_i|=2i,\\
    \label{eq:CohomologyOfBT}
    &H^*(BT) \cong \F_3[\tau_1,\cdots,\tau_4], &|\tau_i|=2,\\
    \label{eq:CohomologyOfBK}
    &H^*(BK) \cong \F_3[\xi_1,\xi_2,\xi_3]\otimes\Lambda_{\F_3}[u_1,u_2,u_3],&|\xi_i|=2,\ |u_i|=1.
\end{align}
The element $c_i$ in (\ref{eq:CohomologyOfBU4}) is the $i$-th Chern class modulo 3. The last isomorphism (\ref{eq:CohomologyOfBK}) follows from $K\cong (\Z/3\Z)^3$. Now we will describe the three maps in (\ref{eq:FFactor}). The first map $H^*(B\PU_4)\to H^*(B\U_4)$ was already determined in Proposition \ref{pro:n=4generators}. The second map $H^*(B\U_4)\to H^*(BT)$ takes $c_i$ to the $i$-th elementary symmetric polynomial $\sigma_i$ in variables $\tau_j$'s. The third map $H^*(BT)\to H^*(BK)$ takes $\tau_i\mapsto \xi_i$ for $i=1,2,3$ and $\tau_4\mapsto 0$ as we can see from the explicit embedding $K\hookrightarrow T$  defined in \eqref{eq:K27 definition}. Hence, if we compose the three maps in (\ref{eq:FFactor}), we conclude that 
    \begin{align*}
    \varphi^* (\epsilon_4)&=\sigma_{2}\\
    \varphi^* (\epsilon_6)&=
    \sigma_{1}^{3} -\sigma_{1}\sigma_{2}-\sigma_{3}\\
    \varphi^* (\epsilon_8)&=
    \sigma_{1}\sigma_{3}-\sigma_{1}^{2}\sigma_{2}
    \end{align*}
where $\sigma_i$ now denotes the $i$-th elementary symmetric polynomials in the variables $\xi_1,\xi_2,\xi_3$ which are elements in $H^*(BK)$.

    Next, in order to apply Proposition \ref{pro:PUn/Gamma}, we need to show that 
    \[(\sigma_2,\ \sigma_1^3 - \sigma_1\sigma_2 - \sigma_3,\ \sigma_1\sigma_3 - \sigma_1^2\sigma_2)\] 
    forms a regular sequence in $\F_3[\xi_1,\xi_2,\xi_{3}]$ regarded as a subring of $H^*(BK;\F_3)$.   We first need the following two general results about regular sequences.     
\begin{lem}[\cite{eisenbud2013commutative}, Corolary 17.2]\label{lem:regular_permute}
    If $R$ is a Noetherian local ring and $x_1,\cdots,x_r$ is a regular sequence, of elements in the maximal ideal of $R$, then any permutation of $x_1,\cdots,x_r$ is again a reguler sequence.
\end{lem}

\begin{lem}\label{lem:local_regular}
    Let $\F$ be a field. Let $R = \F[x_1,\cdots,x_n]$. Let $\mathfrak{m}$ be the maximal ideal generated by $x_1,\cdots,x_n$. Let $f_1,\cdots,f_k$ be  homogeneous polynomials in $\mathfrak{m}$. Then $(f_1,\cdots,f_k)$ is a regular sequence in the local ring $R_{\mathfrak{m}}$ if and only if it is a regular sequence in $R$.
\end{lem}

\begin{proof}
The ``only if'' direction is straightforward because $R_{\fm}$ is an integral domain and $R$ is one of its subrings. We just need to prove the ``if'' direction.  Let $(f_1,\cdots,f_k)$ be a regular sequence in $R$. Let $I_j$ be the homogeneous ideal generated by $(f_1,\cdots,f_{j-1})$. If we have a homogeneous polynomial $g\in R$ such that in $R_{\fm}$, we have $gf_j\in (I_j)_{\fm}$, then we have an $h\in\fm$ satisfying
    \begin{equation*}
        (1+h)gf_j\in I_j.
    \end{equation*}
    Since $I_j$ is homogeneous, we have $gf_j\in I_j.$
\end{proof}

    By Lemma~\ref{lem:local_regular} and Lemma~\ref{lem:regular_permute}, a regular sequence of homogeneous elements in $\F_3[\xi_1,\xi_2,\xi_3]$ is still regular after a permutation of its order. We start by checking by hand that $(\sigma_{1},\sigma_{2},
    \sigma_{3})$ is a regular sequence. Therefore, $(\sigma_{2},\sigma_{1},
    \sigma_{3})$ is also a regular sequence, and then so is 
    $(\sigma_2,\ \sigma_1,\ \sigma_1^3 - \sigma_3)$.
    We permute them again to obtain 
    \[(\sigma_2,\ \sigma_1^3 - \sigma_3,\ \sigma_1),\]
    and then
    \begin{equation*}
        (\sigma_2,\ \sigma_1^3 - \sigma_1\sigma_2 - \sigma_3,\ \sigma_1),\quad (\sigma_2,\ \sigma_1^3 - \sigma_1\sigma_2 - \sigma_3,\ \sigma_1\sigma_3),
    \end{equation*}
    and finally
    \[(\sigma_2,\ \sigma_1^3 - \sigma_1\sigma_2 - \sigma_3,\ \sigma_1\sigma_3 - \sigma_1^2\sigma_2).\]
    
    Since $(\varphi^*(\epl_4), \varphi^*(\epl_6), \varphi^*(\epl_8))$ is a regular sequence in $\F_3[\xi_1,\xi_2,\xi_3]$ regarded as a subring of $H^*BK$, we apply Proposition \ref{pro:PUn/Gamma} to conclude that the classifying map $cl:\PU_4/K\to BK$ of the normal $K$-cover $\PU_4\to \PU_4/K$ induces a surjective map 
    $$cl^*:H^*(BK)\twoheadrightarrow H^*(\PU_4/K).$$
    In particular, $cl^*$ is nonzero in the top dimension $\dim(\PU_4/K)=15$. By Proposition \ref{pro:homo genus}, we conclude that $g(\PU_4\to\PU_4/K)\ge 16.$
\end{proof}

\section{The 24 inflection points and 28 bitangents on quartic curves}
This section is devoted to proving Theorem \ref{SG bitangent} and Theorem \ref{SG flex}.
Recall that we defined
\begin{align*}
    \Bb&:=\{\text{nonsingular homogeneous quartic polynomials } F(x,y,z)\}/\C^\times\\
    \Eb&:=\{(F,l_1,\cdots, l_{28}) :  \text{ $l_i$'s are the 28 distinct bitangent lines to $F$}\}
\end{align*}
The natural map $\Eb\to \Bb$ given by $(F,l_1,\cdots, l_{28})\mapsto F$ is a normal $S_{28}$-cover. This cover is not connected. Its monodromy representation $\pi_1(\Bb)\to S_{28}$ is not surjective with an image being a proper subgroup of $S_{28}$ isomorphic to the Steiner group $O_6(\Z/2)$. See e.g. \cite{Harris} for more details. Similarly, recall that
\begin{align*}
    \Bf&:=\{F(x,y,z) : F \text{ defines a quartic curve with 24 distinct inflection points}\}/\C^\times\\
    \Ef&:=\{(F,p_1,\cdots, p_{24}) :  \text{ $p_i$'s are the 24 distinct inflection points in $F$}\}
\end{align*}
Since a generic quartic curve has 24 distinct inflection points, $\Bf$ is an open subspace of the parameter space $\CP^{14}$ of all quartic curves. The projection $\Ef\to \Bf$ given by $(F,p_1,\cdots, p_{24})\mapsto F$ is a normal $S_{24}$-cover. This cover is connected by a result of Harris (page 698 in \cite{Harris}). $\PU_3$ acts naturally on $\CP^2$ and hence also on $\Bb,\Eb,\Bf,$ and $\Ef$. The covering maps $\Eb\to\Bb$ and $\Ef\to\Bf$ are both $\PU_3$-equivariant.

Let us abbreviate
$$\zeta:=e^{2\pi i/7}\qquad\text{ and }\qquad \alpha:=1+\zeta^2+\zeta^4 = (-1+\sqrt{-7})/2.$$
Consider the quartic curve defined by 
\begin{equation}
    \label{eq: Klein quartic 2}
    F(x,y,z) = x^4+y^4+z^4+3\alpha(x^2y^2+y^2z^2+z^2x^2)=0.
\end{equation}
This curve is projective equivalent to the famous Klein quartic curve defined by 
$$x^3y+y^3z+z^3x=0$$
after a change of coordinates give by the matrix
$$\begin{bmatrix}
     1 & 1+\zeta\alpha & \zeta^2+\zeta^6 \\
    1+\zeta\alpha & \zeta^2+\zeta^6 & 1 \\
    \zeta^2+\zeta^6 & 1 & 1+\zeta\alpha    
\end{bmatrix}.$$
We will refer to $F$ in \eqref{eq: Klein quartic 2} as the Klein quartic curve even though it is in different coordinates. The Klein quartic curve is nonsingular and is best known for having the biggest automorphism group among all genus-3 curves. 

\begin{prop}
\label{pro:Klein auto on flex}
The automorphism group $G_{168}$ of the Klein quartic curve is isomorphic to $\mathrm{PSL}_2(\F_7)$, which is a simple group of order 168. This curve has 24 distinct inflection points and 28 distinct bitangent lines. The actions of $G_{168}$ on the set of 24 inflection points and on the set of 28 bitangents are both transitive. The induced homomorphisms
$$G_{168}\to S_{24}\qquad\text{ and }\qquad G_{168}\to S_{28}$$
are both injective. 
\end{prop}
A proof of the proposition above can be found on p.64 in \cite{elkieskleinquartic}.

Consider the following subgroup $H\le G_{168}$ preserving the defining equation of Klein quartic in \eqref{eq: Klein quartic 2}:
\begin{equation}
\label{eq:def of H acting on Klein quartic}
    H:= \Bigg\langle
\begin{bmatrix}
     -1 & 0 & 0 \\
    0 & 1 & 0 \\
    0 & 0 & 1    
\end{bmatrix},\ \ 
\begin{bmatrix}
    1 & 0 & 0 \\
    0 & -1 & 0 \\
    0 & 0 & 1    
\end{bmatrix}
\Bigg\rangle.
\end{equation}
$H$ is isomorphic to $(\Z/2\Z)^2$. The $\PU_3$-orbit of the Klein quartic   gives well-defined maps 
\begin{align}
\label{eq:orbit map flex}
\theta: \PU_3/H\to \Bf 
\qquad\text{ and }\qquad
\theta: \PU_3/H\to \Bb \qquad \text{ where } \qquad\theta(g):=gF.
\end{align}
We use the same $\theta$ to denote the two maps above because they are the same map only with different codomains. 

\begin{prop}
\label{pro:pullback C disconnected}
The total space of the pullback of the normal $S_{28}$-cover $\Eb\to \Bb$ along the map $\theta: \PU_3/H\to \Bb$ is a disjoint union of homeomorphic copies of $\PU_3$
$$\theta^*\Eb\cong \bigcup_{S_{28}/H}\PU_3$$
where the components are in bijection with $H$-cosets in $S_{28}$. Here we identify $H$ as a subgroup of $S_{28}$ using Proposition \ref{pro:Klein auto on flex}.

Similarly, total space of the pullback of the normal $S_{24}$-cover $\Ef\to \Bf$ along the map $\theta: \PU_3/H\to \Bf 
$ is a disjoint union of homeomorphic copies of $\PU_3$
$$\theta^*\Ef\cong \bigcup_{S_{24}/H}\PU_3$$
where the components are in bijection with $H$-cosets in $S_{24}$. Again we identify $H$ as a subgroup of $S_{24}$ using Proposition \ref{pro:Klein auto on flex}.
\end{prop}
\begin{proof}
    By definition of pullback,
    $$\theta^*\Ef = \{(gH,p_1,...,p_{24}) : gH\in\PU_3/C \text{ and $p_i$'s are the 24  inflection points on the curve $gF$}\}.$$
    This space admits natural actions by $\PU_3$  where $h\cdot (gC,p_1,...,p_{24}) = (hgC,hp_1,...,hp_{24})$. This action is free by Proposition \ref{pro:Klein auto on flex}. Hence, $\theta^*\Ef$ is a disjoint union of $\PU_3$-orbits, where each orbit is a copy of $\PU_3$. Since $\theta^*\Ef$ is an $S_{24}$-cover of $\PU_3/H$, the components of $\theta^*\Ef$ must be in bijection with $S_{24}/H$. 

    The proof of the statement for $\theta^*\Eb$ is completely analogous. 
\end{proof}
Now we have
\begin{align*}
    g(\Ef\to\Bf)&\ge g(\theta^*\Ef\to\PU_3/H)    &\text{by Proposition \ref{pro:pullback SG}}\\
    &=g(\PU_3\to \PU_3/H)&\text{by Proposition \ref{pro:pullback C disconnected} and Proposition \ref{pro:disconnected covers}}
\end{align*}
\begin{align*}
    g(\Eb\to\Bb)&\ge g(\theta^*\Eb\to\PU_3/H)    &\text{by Proposition \ref{pro:pullback SG}}\\
    &=g(\PU_3\to \PU_3/H)&\text{by Proposition \ref{pro:pullback C disconnected} and Proposition \ref{pro:disconnected covers}}
\end{align*}
The proof of Theorem \ref{SG bitangent} and Theorem  \ref{SG flex} will be complete after we prove the following proposition:

\begin{prop}
\label{pro: genus H}
    $g(\PU_3\to\PU_3/H)=9.$
\end{prop}
\begin{proof}
    The upper bound $g(\PU_3\to\PU_3/H)\le 9$ follows from  Corollary \ref{cor:coho dim} together with the fact that $\PU_3/H$ is a compact manifold of dimension 8. We focus on the lower bound.

    In this proof we will only use cohomology with $\F_2$-coefficients and thus will suppress $\F_2$ in notation. As in previous sections, let $\varphi:BH\to B\PU_3$ denote the map induced by the inclusion $H\hookrightarrow \PU_3$. Notice that the inclusion $H\hookrightarrow \PU_3$ is a composition of
    $H\hookrightarrow T\hookrightarrow \U_3\twoheadrightarrow \PU_3$ where $T$ denotes a maximal torus of $\U_3$ consisting of matrices with 0's off diagonal. Therefore, we have
    \begin{equation}
    \label{eq:FFactorH}
    \begin{tikzcd}
	H^*(B\PU_3) & H^*(B\U_3) & H^*(BT) & H^*(BH)
	\arrow[from=1-1, to=1-2]
	\arrow["\varphi^*"', bend right=10, from=1-1, to=1-4]
	\arrow[from=1-2, to=1-3]
	\arrow[from=1-3, to=1-4]
    \end{tikzcd}
    \end{equation}
        the composition of which is exactly $\varphi^*$. Proposition \ref{pro:n=3generators} gave 
    \begin{equation}
        \label{eq:CohomologyOfBPU3}
        H^*(B\PU_3)\cong \F_2[\epsilon_4,\epsilon_6],\qquad\qquad\qquad
        |\epsilon_i|=i.
    \end{equation}
    We also have the following standard results
\begin{align}
    \label{eq:CohomologyOfBU3 mod 2}
    &H^*(B\U_3) \cong \F_2[c_1,c_2,c_3], &|c_i|=2i,\\
    \label{eq:CohomologyOfBT3}
    &H^*(BT) \cong \F_2[\tau_1,\tau_2,\tau_3], &|\tau_i|=2,\\
    \label{eq:CohomologyOfBH}
    &H^*(BH) \cong \F_2[u,v],& |u|=|v|=1.
\end{align}
The element $c_i$ in (\ref{eq:CohomologyOfBU3 mod 2}) is the $i$-th Chern class modulo 2. The last isomorphism (\ref{eq:CohomologyOfBH}) follows from $H\cong (\Z/4\Z)^2$. Now we will describe the three maps in (\ref{eq:FFactorH}). The first map $H^*(B\PU_3)\to H^*(B\U_3)$ was already determined in Proposition \ref{pro:n=3generators}. The second map $H^*(B\U_3)\to H^*(BT)$ takes $c_i$ to the $i$-th elementary symmetric polynomial $\sigma_i$ in variables $\tau_j$'s. The third map $H^*(BT)\to H^*(BH)$ takes $\tau_1\mapsto u^2$ and $\tau_2\mapsto v^2$ and $\tau_3\mapsto 0$ as we can see from the explicit embedding $H\hookrightarrow T$  defined in \eqref{eq:def of H acting on Klein quartic}. Hence, if we compose the three maps in (\ref{eq:FFactorH}), we conclude that 
    \begin{align}
    \label{phi eps4,6}
    \varphi^* (\epsilon_4)&
    = u^4+v^4+u^2v^2= (u^2+uv+v^2)^2\\
    \nonumber
    \varphi^* (\epsilon_6)&
    =u^2v^2(u^2+v^2) = u^2v^2(u+v)^2.
    \end{align}
The sequence $(\varphi^*(\epl_4),\varphi^*(\epl_6))$ forms a regular sequence in $H^*(BH)$ in the sense of Definition \ref{def:regular seq} because $H^*(BH)\cong \F_2[u,v]$ is a unique factorization domain and the two elements $\varphi^*(\epl_4)$ and $\varphi^*(\epl_6)$ in \eqref{phi eps4,6} are relatively prime. Therefore, by Proposition \ref{pro:PUn/Gamma}, we conclude that the classifying map $cl:\PU_3/H\to BH$ of the normal $H$-cover $\PU_3\to \PU_3/H$ induces a surjective map 
$$cl^*:H^*(BH)\twoheadrightarrow H^*(\PU_3/H).$$
In particular, $cl^*$ is nonzero in the top dimension $H^8$. By Proposition \ref{pro:homo genus}, we conclude that $g(\PU_3\to\PU_3/H)\ge 9$.
\end{proof}

\section{Proof of Theorem \ref{TC main theorems}}
We first prove part (1) of Theorem \ref{TC main theorems} which states that the topological complexity of the problem $\linepro$ is at least $15$ for any $\epsilon$ small enough. Suppose that there exists an algorithm with $k$ many output leaves (and hence $k-1$ many branching nodes) that solves the problem $\linepro$. Our goal is to show that $k-1\ge 15$. 

For $i=1,\cdots,k$, let $A_i$ denote the subset of $\Bl$ consisting of all those inputs that will arrive at the $i$-th output leaf of the given algorithm. Then $\Bl$ as a set is a disjoint union of $A_i$'s. Each $A_i$ is locally closed because it is the intersection of finitely many closed subsets consisting of inputs satisfying $h\le 0$ at those branching nodes where the path from the input root to the $i$-th output leaf turns to the right, and finitely many open subsets consisting of inputs satisfying $h> 0$ at those branching node where the same path turns to the left.  Moreover, the algorithm gives a map $\psi_i:A_i\to G(2,4)^{\times 27}$, where $G(2,4)$ denotes the space of all lines in $\CP^3$, such that for any $F\in A_i$, we have $\psi_i(F)=(l'_1,\cdots,l'_{27})$ where each $l'_i$ is  $\epsilon$-close to a line $l_i$ where $l_i$'s are the 27 distinct lines on the cubic surface $F$. By the Tietze Extension Theorem, $\psi_i$ can be extended to an open subset $U_i\supseteq A_i$ satisfying the same property.

Recall that the $\PU_4$-orbit of the Fermat cubic surface $F$ in $\Bl$ gives a well-defined map $\eta:\PU_4/K\to \Bl$ as defined in (\ref{eq:orbit map lines}). Let
$$\epl_0:=\frac{1}{2}\text{ minimum distance between two distinct lines on $gF$ among all } g\in \PU_4.$$
Since $\PU_4$ is compact, $\epl_0$ is strictly positive. Let $V_i:=\eta^{-1} U_i$. By our construction, for any $gK\in V_i$, the point $\psi_i(\eta(gK))$ is $\epsilon$-close to a line in $\eta(gK)=g\cdot F$. Moreover, this line is unique if $\epsilon<\epsilon_0$. Let this unique line be denoted by $s_i(gK)$. In summary, we have open sets $W_i$ for $i=1,\cdots,k$ that form an open cover of $\PU_4/K$ such that on each $W_i$, there exists a continuous section $s_i:W_i\to \eta^*\El$ where $\eta^*\El$ is the pullback of $\El\to\Bl$ along $\eta$. By Propositions \ref{pro:disconnected covers}, \ref{pro:pullback K disconnected}  and \ref{pro: genus K27}, it must be that $$k\ge g(\eta^*\El\to \PU_4/K)=g(\PU_4\to\PU_4/K)= 16.$$ 

The proofs of part (2) and (3) are the same as the proof of (1) above. We use Proposition \ref{pro: genus H} in place of Proposition \ref{pro: genus K27}. 
\qed

\bibliographystyle{abbrv}
\bibliography{main.bib}

\end{document}